\documentclass{article}
\usepackage[latin1]{inputenc}
\usepackage{amsfonts}
\usepackage{amsthm}
\usepackage{amsmath}\interdisplaylinepenalty=2500
\usepackage{amssymb}
\usepackage{nccmath}
\usepackage[english]{babel}
\usepackage{hyperref}
\usepackage[all]{xy}
\usepackage{cite}
\usepackage{url}

\newcommand{\U}{{\mathcal U}}
\renewcommand{\r}{\rightarrow}

\newtheorem{lemma}{Lemma}
\newtheorem{example}[lemma]{Example}
\newtheorem{definition}[lemma]{Definition}

\newtheorem{theorem}{Theorem}
\newtheorem{remark}[lemma]{Remark}
\newtheorem{notation}[lemma]{Notation}

\newcommand{\Ctx}{\mathrm{Ctx}}

\newcommand{\Ty}{\mathrm{Ty}}
\newcommand{\Tm}{\mathrm{Tm}}
\newcommand{\Sub}{\mathrm{Sub}}
\newcommand{\Hom}{\mathrm{Hom}}
\newcommand{\id}{\mathrm{id}}

\newcommand{\El}{\mathrm{El}}
\newcommand{\app}{\mathrm{app}}

\newcommand{\Alg}{\mathrm{Alg}}

\renewcommand{\_}{\rule{.6em}{.5pt}\hspace{0.023cm}}
\DeclareMathSymbol{:}{\mathbin}{operators}{"3A}

\newcommand{\C}{{\mathcal C}}
\newcommand{\D}{{\mathcal D}}
\newcommand{\E}{{\mathcal E}}
\newcommand{\I}{{\mathcal I}}
\newcommand{\colim}{\mathrm{colim}}
\renewcommand{\lim}{\mathrm{lim}}
\newcommand{\fin}{\mathrm{fin}}
\newcommand{\Lex}{\mathrm{Lex}}

\newcommand{\Set}{\mathrm{Set}}

\newcommand{\wk}{\mathrm{w}}
\newcommand{\var}{\mathrm{v}}

\renewcommand{\t}{\mathrm{tt}}
\newcommand{\eqd}{\smash{\overset{d}{=}}}
\newcommand{\eqi}{\smash{\overset{i}{=}}}
\newcommand{\eqm}{\smash{\overset{m}{=}}}
\newcommand{\eqf}{\smash{\overset{f}{=}}}

\newcommand{\TT}{\mathrm{TT}}
\newcommand{\pTT}{\mathrm{pTT}}

\newcommand{\emptyctx}{
\boldsymbol{\cdot}}

\author{Hugo Moeneclaey\\
Universit\'e de Paris, \\
Inria Paris, CNRS, IRIF, \\
France\\
}

\title{Parametricity and Semi-Cubical Types}

\begin{document}


\maketitle

\begin{abstract}
We construct a model of type theory enjoying parametricity from an arbitrary one. A type in the new model is a semi-cubical type in the old one, illustrating the correspondence between parametricity and cubes. 

Our construction works not only for parametricity, but also for similar interpretations of type theory and in fact similar interpretations of any generalized algebraic theory. To be precise we consider a functor forgetting unary operations and equations defining them recursively in a generalized algebraic theory.  We show that it has a right adjoint.

We use techniques from locally presentable category theory, as well as from quotient inductive-inductive types.

\end{abstract}

\tableofcontents

\section{Overview}

\subsection{Introduction to type theory}

Martin L{\"o}f type theory \cite{martin1984intuitionistic} is a foundational system for constructive mathematics. Most modern proof assistants (for example Agda, Coq, Lean) are based on variants of this system. 

In this paper we use a semantical approach: we study models of type theory rather than type theory itself. Such a model supplies suitable notions of types (corresponding to sets and properties) and terms inhabiting them (corresponding to elements of sets and proofs of propositions). 
It can be intuitively conceived as a world where (constructive) mathematics can take place. We give two basic models:
\begin{itemize}
\item The set model is the traditional mathematical world.
\item The initial model is inhabited by the elements definable from the syntax of type theory. Its elements and identifications are precisely the ones which can be derived from the axioms of type theory.
\end{itemize}
Proof assistants implement the initial model. This makes possible a rich interaction between models of type theory and proof assistants, mainly using two principles: 
\begin{itemize}
\item On the one hand a given model has features absent from the initial one. They can suggest extensions for a proof assistant.
\item On the other hand a term in the initial model can be interpreted in any model. So a formal proof gives a multitude of theorems, one for each model.
\end{itemize}
The abundance of models for type theory makes this approach fruitful. We give an example for each principle.

Perhaps the most striking use of the first principle comes from the Kan cubical set model given in \cite{bezem2014model}. This model was used to design cubical type theory~\cite{cohen2015cubical}, which is now implemented as Cubical Agda \cite{vezzosi2019cubical}. This proof assistant based on Agda supports (among other features) the axiom of univalence, implying that isomorphic types are equal. This axiom has applications in computer science, where it allows to transport programs along isomorphisms, and in higher mathematics (that is mathematics with everything considered up to homotopy).

Now we give an example for the second principle. Schreier theory classifies group extensions. It can be proven in type theory with univalence using an alternative definition of groups as pointed connected types (a sketch is given by David Myers at \url{http://davidjaz.com/Talks/DJM_HoTT2020.pdf}). But type theory with univalence can be interpreted in any higher topos \cite{shulman2019all}, so we get a higher Schreier theory classifying group extensions in higher topoi. But the type-theoretic proof is significantly easier than the usual proof for regular Schreier theory, let alone a higher version! The canonical introduction to such synthetic homotopy theory using univalence is called the HoTT book \cite{hottbook}.

\subsection{Parametricity and cubical structure}

Now we introduce a fundamental tool of type theory called parametricity. Originally designed for system $F$ \cite{reynolds1983types}, a generalization to type theory can be found in \cite{bernardy2010parametricity}. From our semantical point of view, parametricity consists of some operations $\_^*$ defined recursively in the initial model. Using these operations we can extract \emph{Theorems for free!} \cite{wadler1989theorems} for polymorphic terms in the initial model (that is terms taking a type as input). For example given any term $t$ in the initial model of the same type as the polymorphic identity, the term $t^*$ proves that $t$ behaves as the polymorphic identity. Parametricity can be summarized by saying that for any term $t$ in the initial model, $t^*$ proves that $t$ preserves relations.

We define a parametric model as a model of type theory together with operations $\_^*$ showing that terms in this model preserve relations. The main goal of this article is to construct a parametric model from an arbitrary one. 


Now we introduce the seemingly unrelated semi-cubical sets. A semi-cubical set consists of a set of points together with:
\begin{itemize}
\item For any two points, a set of paths between them.
\item For any four points and four paths drawing a square, a set of surfaces having this square as border.
\item And so on in higher dimensions (cubes, hypercubes...).
\end{itemize}
The geometric intuition should be clear from our vocabulary. A cubical set is a semi-cubical set with degeneracies, meaning constant paths, constant surfaces, etc. 

It is known that there is a strong connection between cubical structures and parametricity. In \cite{atkey2014relationally} a model for non-iterated parametricity (i.e. parametricity where $\_^*$ can only be applied once) by $1$-truncated cubical sets (i.e. reflexive graphs) is given. In \cite{ghani2016proof} a model for the so-called proof-relevant parametricity (where $\_^*$ can be applied twice) is given using $2$-truncated semi-cubical sets. Alternatively \cite{sojakova2018general} gives a general, foundation-agnostic, definition of a model for non-iterated parametricity in system F using reflexive graphs, and \cite{johann2017cubical} extends this to full parametricity using cubical structures.

Moreover it is known that cubical structures arise when trying to internalize parametricity in type theory. For example \cite{bernardy2015presheaf} gives a model for (the unary variant of) parametricity in (the unary variant of) cubical sets, and \cite{cavallo2020internal} shows how parametricity can be internalized orthogonally to univalence, using very similar cubical techniques for both features.

In this paper we will show how semi-cubical structures arise already from iterated parametricity without internalization, and can be used to construct a parametric model from any given model.

\subsection{Content of this paper}

We use the language of categories. 
There is a functor from parametric models to arbitrary models, which forgets the unary operations $\_^*$. This forgetful functor has a poorly behaved left adjoint freely adding parametricity, sending non-parametric models to trivial ones. In this paper we show this forgetful functor also has a better-behaved right adjoint. We will see that this right adjoint sends a model $\C$ to the model of semi-cubical types in~$\C$. For example it sends the set model to the semi-cubical set model. Moreover we will see that our method for constructing this right adjoint does not use much features of parametricity. From these two facts a picture emerges: from an interpretation of type theory (for example parametricity), we get:
\begin{itemize} 
\item A notion of structure on types (for example semi-cubical types).
\item Models for this interpretation by types with this structure (for example semi-cubical models for parametricity).
\end{itemize}

Recall there is a link between cubical structure and internal parametricity. In fact we pursue the idea that univalence is a variant of parametricity, and that we have a similar link between Kan cubical structure and univalence. This is already evoked in \cite{altenkirch2018towards} which introduces a syntax for a univalent type theory inspired by parametricity but does not prove univalence, and in \cite{tabareau2018equivalences} which unifies parametricity and univalence, although assuming a univalent universe to begin with. To summarize we conjecture a table:
\begin{center}
\begin{tabular}{|c|c|c|}
\hline
Interpretation & Structure \\
\hline
External parametricity & Semi-cubical types\\
Internal parametricity & Cubical types\\
External univalence & Kan semi-cubical types\\
Internal univalence & Kan cubical types\\
\hline
\end{tabular}
\end{center}

In this paper we give the procedure supposedly linking the two columns, and study the first line in detail. It is organized as follows:
\begin{itemize}
\item In Section \ref{sectionDefinitionParametricity} we define parametricity for models of type theory. Then we give a general definition of an interpretation, give toy examples, and prove that parametricity is an interpretation. It should be noted that we use unary parametricity (which can also be seen as a form of realizability as first noted in \cite{bernardy2011realizability}) for notational convenience, so we obtain something a bit different from semi-cubes. But our approach can be extended to binary parametricity and semi-cubes straightforwardly.
\item In Section \ref{generalCategories} we give general conditions implying the existence of right adjoints to forgetful functors for extensions of generalized algebraic theories. To do this we use locally presentable categories, which generalize categories of models for an algebraic theory. The standard textbook is \cite{adamek1994locally}. This short section contains only well-known material, but we include it anyway because the intended audience for this paper might not be familiar with it.
\item In Section \ref{constructingRightAdjoint} we prove our main result, constructing a right adjoint from any interpretation. We examine this adjoint for our toy examples of  interpretation to help build intuition, and then for parametricity to see it constructs semi-cubical models. 
\end{itemize}

\begin{remark}
Three very similar syntactic notions can be used as a basis for the general definition of an interpretation. 
\begin{itemize}
\item Essentially algebraic theories, giving rise to locally presentable categories \cite{adamek1994locally}.
\item Generalized algebraic theories, allowing dependencies \cite{cartmell1986generalised}.
\item Signatures for quotient inductive-inductive types, where there is a built-in strong induction principle for initial objects \cite{kaposi2019constructing}.
\end{itemize}
Technically we define interpretations using generalized algebraic theories, but we freely use the most convenient of these three notions throughout this paper, as they are intuitively equivalent. This should not be interpreted as a claim that these three notions are equivalent in a technical sense, as we do not know any satisfying reference for this, and it is out of scope to prove it here. However our main example (parametricity as an interpretation of type theory) can indeed be defined using any of these three syntactic notions.
\end{remark}




\section{Parametricity as an interpretation}

\label{sectionDefinitionParametricity}

\subsection{Models of type theory}

We define a model of type theory as a CwF (Category with Families \cite{dybjer1995internal}), so they are algebras for a GAT (Generalized Algebraic Theory in Cartmell's sense \cite{cartmell1986generalised}). We follow the presentation in \cite{altenkirch2016type}, which explains how such a GAT can be seen as a signature for a QIIT (Quotient Inductive-Inductive Type). Indeed the induction principle for the initial model of type theory seen as a QIIT will be crucial in our construction.

First we define CwFs. In this definition $\Ctx$ stands for contexts, $\Ty$ for types, $\Sub$ for substitutions and $\Tm$ for terms. We use type-theoretic notations, so a set-minded reader should replace $x:A$ by $x\in A$ and $p\ q$ by $p(q)$. Moreover any variable appearing undeclared is in fact universally quantified. We use Agda notations, meaning we use:
\begin{eqnarray}
(x:A)\r B(x)
\end{eqnarray}
for what is usually denoted $\forall x:A, B(x)$ or $\Pi(x:A).B(x)$. 

\begin{definition}
A CwF consists of:
\begin{eqnarray}
\Ctx &:& \Set \\
\Ty &:& \Ctx \r \Set \\
\Sub &:& \Ctx \r \Ctx \r \Set \\
\Tm &:& (\Gamma : \Ctx)\r \Ty\ \Gamma \r \Set 
\end{eqnarray}

With constructors for contexts:
\begin{eqnarray}
\emptyctx &:& \Ctx \\
(\_,\_) &:& (\Gamma:\Ctx) \r \Ty\ \Gamma \r \Ctx
\end{eqnarray}

types:
\begin{eqnarray}
\_[\_] &:& \Ty\ \Delta \r \Sub\ \Gamma\ \Delta \r \Ty\ \Gamma
\end{eqnarray}

substitutions:
\begin{eqnarray}
\_\circ\_ &:& \Sub\ \Delta\ \Theta \r \Sub\ \Gamma\ \Delta \r \Sub\ \Gamma\ \Theta \\
\id &:& \Sub\ \Gamma\ \Gamma \\
\epsilon &:& \Sub\ \Gamma\ \emptyctx \\
(\_,\_) &:& (\delta : \Sub\ \Gamma\ \Delta) \r \Tm\ \Gamma\ (A[\delta]) \nonumber \\
 & & \r \Sub\ \Gamma\ (\Delta,A) \\
\pi_1 &:& \Sub\ \Gamma\ (\Delta,A)\r \Sub\ \Gamma\ \Delta
\end{eqnarray}

and terms:
\begin{eqnarray}
\_[\_] &:& \Tm\ \Delta\ A \r (\delta : \Sub\ \Gamma\ \Delta) \nonumber \\
& & \r \Tm\ \Gamma\ (A[\delta]) \\
\pi_2 &:& (\sigma : \Sub\ \Gamma\ (\Delta,A)) \r \Tm\ \Gamma\ (A[\pi_1\ \sigma])
\end{eqnarray}

with the following equations for types:
\begin{eqnarray}
A[\sigma\circ\eta] &=& A[\sigma][\eta] \\
A[\id] &=& A 
\end{eqnarray}

substitutions:
\begin{eqnarray}
(\sigma\circ\nu)\circ\delta &=& \sigma\circ(\nu\circ\delta)\\
\id\circ\sigma &=& \sigma\\
\sigma\circ\id &=& \sigma\\
\sigma &=& \epsilon\\
\pi_1\ (\sigma,t) &=& \sigma\\
(\pi_1\ \sigma,\pi_2\ \sigma) &=& \sigma\\
(\sigma,t)\circ\nu &=& (\sigma\circ\nu,t[\nu])
\end{eqnarray}

and terms:
\begin{eqnarray}
\pi_2\ (\sigma,t) &=& t
\end{eqnarray}

\end{definition}

Note that some equations in the definition need the previous ones to be well-typed. CwFs are worlds where one can substitute, so they are sometimes called models for the calculus of substitutions. But not much can be done in an arbitrary CwF since there is no way to construct types. 
We define additional structures on a CwF, which will be assumed in a model of type theory.

Now we introduce some useful notations where $\wk$ stands for weakening and $\var$ for the last variable. 

\begin{notation}
We define for $\Gamma:\Ctx$ and $A:\Ty\ \Gamma$.
\begin{eqnarray}
\wk = \pi_1\ \id &:& \Sub\ (\Gamma,A)\ \Gamma\\
\var = \pi_2\ \id &:& \Tm\ (\Gamma,A)\ A[\wk]
\end{eqnarray}
\end{notation}

Using this notation we have that $\var[\wk^n]$ is similar to the de Bruijn index $n$,  where $\wk^n$ is $\wk\circ\dots\circ\wk$ where $\wk$ is composed $n$ times.

\begin{definition}
A unit for a CwF consists of:

\begin{eqnarray}
\top &:& \Ty\ \Gamma \\
\t &:& \Tm\ \Gamma\ \top
\end{eqnarray}

such that for all $x:\Tm\ \Gamma\ \top$  we have:

\begin{eqnarray}
x &=& \t
\end{eqnarray}

with equations for substitutions:

\begin{eqnarray}
\top[\sigma] &=& \top\\
\t[\sigma] &=& \t
\end{eqnarray}
\end{definition}


\begin{definition}

Products for a CwF consist of:
\begin{eqnarray}
\Sigma  &:&(A:\Ty\ \Gamma) \r \Ty\ (\Gamma,A)\r \Ty\ \Gamma\\
\_.1 &:& \Tm\ \Gamma\ (\Sigma\ A\ B) \r \Tm\ \Gamma\ A\\
\_.2 &:& (t:\Tm\ \Gamma\ (\Sigma\ A\ B)) \r \Tm\ \Gamma\ B[\id,t.1]\\
(\_,\_) &:& (t:\Tm\ \Gamma\ A) \r \Tm\ \Gamma\ B[\id,t]\nonumber\\
 & & \r \Tm\ \Gamma\ (\Sigma\ A\ B)
\end{eqnarray}

such that we have:
\begin{eqnarray}
(s,t).1 &=& s\\
(s,t).2 &=& t\\
(t.1,t.2) &=& t
\end{eqnarray}

with equations for substitutions:
\begin{eqnarray}
(\Sigma\ A\ B)[\sigma] &=& \Sigma\ A[\sigma]\ B[\sigma\circ\wk,\var]\\
(s,t)[\sigma] &=& (s[\sigma],t[\sigma])
\end{eqnarray}
\end{definition}

\begin{remark}
We call $\Sigma\ A\ B$ a product because it specializes to the cartesian product $A\times B$ when $B$ does not depend on $A$. Confusingly it is sometimes called a dependent sum.
\end{remark}

\begin{remark}
The notation $(\_,\_)$ is overloaded, as it can be used for contexts, substitutions or terms.
\end{remark}

\begin{definition}
Functions for a CwF consist of:

\begin{eqnarray}
\Pi &:& (A:\Ty\ \Gamma) \r \Ty\ (\Gamma,A)\r \Ty\ \Gamma\\
\app &:& \Tm\ \Gamma\ (\Pi\ A\ B) \r \Tm\ (\Gamma,A)\ B\\
\lambda &:& \Tm\ (\Gamma,A)\ B \r \Tm\ \Gamma\ (\Pi\ A\ B)
\end{eqnarray}

such that we have:
\begin{eqnarray}
\app\ (\lambda\ t) &=& t\\
\lambda\ (\app\ t) &=& t
\end{eqnarray}

with equations for substitutions:
\begin{eqnarray}
(\Pi\ A\ B)[\sigma] &=& \Pi\ A[\sigma]\ B[\sigma\circ\wk,\var]\\
(\lambda\ t)[\sigma] &=& \lambda\ (t[\sigma\circ\wk,\var])
\end{eqnarray}

\end{definition}

\begin{remark}
The previous definitions imply: 
\begin{eqnarray}
(t.1)[\sigma] &=& t[\sigma].1 \\
(t.2)[\sigma] &=& t[\sigma].2 \\
(\app\ t)[\sigma\circ\wk,\var] &=& \app\ (t[\sigma]) 
\end{eqnarray}
So that there are no missing equations for substitutions.
\end{remark}


\begin{definition}
A universe for a CwF consists of:

\begin{eqnarray}
\U &:& \Ty\ \Gamma\\
\El &:& \Tm\ \Gamma\ \U \r \Ty\ \Gamma\\
\top_\U &:& \Tm\ \Gamma\ \U\\
\Sigma_\U &:& (s : \Tm\ \Gamma\ \U)\r \Tm\ (\Gamma,\El\ s)\ \U \r \Tm\ \Gamma\ \U\\
\Pi_\U &:& (s : \Tm\ \Gamma\ \U)\r \Tm\ (\Gamma,\El\ s)\ \U \r \Tm\ \Gamma\ \U
\end{eqnarray}

such that we have:
\begin{eqnarray}
\El\ \top_\U &=& \top\\
\El\ (\Sigma_\U\ s\ t) &=& \Sigma\ (\El\ s)\ (\El\ t)\\
\El\ (\Pi_\U\ s\ t) &=& \Pi\ (\El\ s)\ (\El\ t)
\end{eqnarray}

with equations for substitutions:
\begin{eqnarray}
\U[\sigma] &=& \U\\
(\El\ t)[\sigma] &=& \El\ (t[\sigma])\\
\top_\U[\sigma] &=& \top_\U\\
(\Sigma_\U\ s\ t)[\sigma] &=& \Sigma_\U\ s[\sigma]\ t[\sigma\circ\wk,\var] \\
(\Pi_\U\ s\ t)[\sigma] &=& \Pi_\U\ s[\sigma]\ t[\sigma\circ\wk,\var]
\end{eqnarray}

\end{definition}


\begin{definition}
A model of type theory is a CwF with a unit, products, functions and a universe.
\end{definition}

There exist many variants of type theory, with fewer or more types. The most common extensions consist in adding some inductive types (for example booleans, natural numbers, identity types, $W$-types...) with sometimes a hierarchy of universes and maybe a scheme for general inductive families. 

Our abstract approach using an interpretation makes it easy to check whether this article works for a given extension. It certainly works for the common ones.

\subsection{Parametric models}

Now we define what it means for a model of type theory to be parametric. We use unary parametricity, which can be seen as a case of realizability. 

\begin{definition}

\label{definitionParametricity}

A parametric model is a model of type theory together with:

\begin{eqnarray}
\_^* &:& (\Gamma:\Ctx)\r \Ty\ \Gamma \label{beginOpParam}\\
\_^* &:& (A:\Ty\ \Gamma) \r \Ty\ (\Gamma,\Gamma^*,A[\wk])\\
\_^* &:& (\sigma : \Sub\ \Gamma\ \Delta) \r \Tm\ (\Gamma,\Gamma^*)\ \Delta^*[\sigma\circ\wk]\\
\_^* &:& (t:\Tm\ \Gamma\ A) \r \Tm\ (\Gamma,\Gamma^*)\ A^*[\id,t[\wk]] \label{endOpParam}
\end{eqnarray}

Such that we have equations for substitutions:
\begin{eqnarray}
\emptyctx^* &=& \top \label{beginEqParam}\\
(\Gamma,A)^* &=& \Sigma\ \Gamma^*[\wk]\ A^*[\wk^2,\var,\var[\wk]]\\
(A[\sigma])^* &=& A^*[\sigma\circ\wk^2,\sigma^*[\wk],\var]\\
(\sigma\circ\nu)^* &=& \sigma^*[\nu\circ\wk,\nu^*]\\
\id^* &=& \var\\
\epsilon^* &=& \t\\
(\sigma,t)^* &=& (\sigma^*,t^*)\\
(\pi_1\ \sigma)^* &=& \sigma^*.1\\
(t[\sigma])^* &=& t^*[\sigma\circ\wk,\sigma^*]\\
(\pi_2\ \sigma)^* &=& \sigma^*.2
\end{eqnarray}

for the unit:
\begin{eqnarray}
\top^* &=& \top\\
\t^* &=& \t
\end{eqnarray}

for products:
\begin{eqnarray}
(\Sigma\ A\ B)^* &=& \Sigma\ A^*[\eta_1]\ B^*[\eta_2]\\ 
(t.1)^* &=& t^*.1\\
(t.2)^* &=& t^*.2\\
(s,t)^* &=& (s^*,t^*)
\end{eqnarray}

where:
\begin{eqnarray}
\eta_1 &=& (\wk,\var.1) \\
\eta_2 &=& (\wk^3,\var.1[\wk],(\var[\wk^2],\var),\var.2[\wk])
\end{eqnarray}

for functions:
\begin{eqnarray}
(\Pi\ A\ B)^* &=& \Pi\ A[\sigma_1]\ (\Pi\ A^*[\sigma_2]\ B^*[\sigma_3]) \\
(\app\ t)^* &=& (\app\ (\app\ t^*))[\nu_1] \\
(\lambda\ t)^* &=& \lambda\ (\lambda\ (t^*[\nu_2]))
\end{eqnarray}

where:
\begin{eqnarray}
\sigma_1 &=& \wk^2 \\
\sigma_2 &=& (\wk^2,\var) \\
\sigma_3 &=& (\wk^4,\var[\wk],(\var[\wk^3],\var),(\app\ \var)[\wk]) \\
\nu_1 &=& (\wk^2,\var.1,\var[\wk],\var.2) \\
\nu_2 &=& (\wk^3,\var[\wk],(\var[\wk^2],\var)) 
\end{eqnarray}

and for the universe:
\begin{eqnarray}
\U^* &=& \Pi\ (\El\ \var)\ \U\\
(\El\ t)^* &=& \El\ (\app\ t^*)\\
(\top_\U)^* &=& \lambda\ \top_\U\\
(\Sigma_\U\ s\ t)^* &=& \lambda\ (\Sigma_\U\ (\app\ s^*)[\eta_1]\ (\app\ t^*)[\eta_2])\\
(\Pi_\U\ s\ t)^* &=& \lambda\ (\Pi_\U\ s[\sigma_1]\ (\Pi_\U\ (\app\ s^*)[\sigma_2] \nonumber\\
 & &  \ \ \ (\app\ t^*)[\sigma_3])) \label{endEqParam}
\end{eqnarray}

\end{definition}

The main point of this lengthy definition is that the equations ensure that the operations $\_^*$ are recursively defined in the initial model by the given equations. This will be checked in detail in the Appendix. 
We will capture this feature in the definition of an interpretation.

\begin{remark}
To treat binary parametricity a product of contexts should be added, so that for $\Gamma:\Ctx$ we can define $\Gamma^* : \Ty\ (\Gamma,\Gamma)$. 
\end{remark}

\begin{remark}
A parametric model needs products and a unit in the equations for $\emptyctx^*$ and $(\Gamma,A)^*$. Moreover a parametric model with a universe needs functions in the equation for $\U^*$. 
\end{remark}

\subsection{Definition of an interpretation}

In this section we give a definition capturing the features of parametricity making the rest of this paper work. We assume the reader familiar with GAT (Generalized Algebraic Theories \cite{cartmell1986generalised}). 

For $T$ a GAT, we denote by $\Alg_T$ the category of models of $T$ and $\I_T$ its initial model, so $\I_T$ is the initial object in $\Alg_T$. Note that we will only consider finitary GATs.


\begin{definition}
Let $T$ be a finitary GAT. An interpretation for $T$ is a GAT of the form: 
\begin{eqnarray}
T,O,E,E'
 \end{eqnarray}
where:
\begin{itemize}
\item $O$ is a set of unary operations defined recursively in $\I_T$ by equations $E$.
\item $E'$ is a set of unary equations proved inductively in $\I_T$.
\end{itemize}
\end{definition}

\begin{remark}
Here a unary operation is an operation with one main input, and possibly secondary inputs inferred from it. 
This means that:
\begin{eqnarray}
(x:A)\r (y : B\ x)\r C
\end{eqnarray}
is considered unary, as $x$ can be inferred from $y$. For example the operations (\ref{beginOpParam}) to (\ref{endOpParam}) are considered unary. They take a context, a type, a term or a substitution as their main input. 

Unary equations are defined similarly as equations depending on one main variable, with possibly secondary variables inferred from it.
\end{remark}

\begin{remark}
\label{remarkCStar}
We clarify the expression `recursively defined in $\I_T$'. It means that for any constructor $c$ in $T$ and any added unary operation $\_^*$ in $T'$, we have an equation in $T'$ of the form:
\begin{eqnarray}
(c(x_1,\dots,x_n))^* &=& c^*(x_1,x_1^*,\dots,x_n,x_n^*) \label{recursiveEquation}
\end{eqnarray}
where $c^*$ is a term in $T$. 

Moreover it means that for any equation $s=t$ in $T$, we have $s^* = t^*$ in $T$, where $s^*$ and $t^*$ are computed using the recursive equations.

Equation \ref{recursiveEquation} makes sense only when there is precisely one unary operation $\_^*$ by sort (as for parametricity), but the analogous general formula is clear.
\end{remark}

\begin{remark}
\label{definitionInterpretation}
This definition can be reformulated more precisely using the theory of signatures for QIITS as in \cite{kovacs2020signatures}. In this theory, given a signature: 
\begin{eqnarray}
\Gamma\vdash 
\end{eqnarray}
we have a signature for displayed algebras: 
\begin{eqnarray}
\Gamma\vdash\Gamma^D
\end{eqnarray} 
and a signature for sections of such a displayed algebra: 
\begin{eqnarray}
\Gamma,\Gamma^D\vdash\Gamma^S
\end{eqnarray} 
Then an interpretation is an extension of $\Gamma$ of the form:
\begin{eqnarray}
\Gamma,\Gamma^S[\id,t]
\end{eqnarray}
 where we have:
\begin{eqnarray}
\Gamma\vdash t:\Gamma^D
\end{eqnarray}
in the theory of signature.
\end{remark}

Interpretations will be used in Section \ref{constructingRightAdjoint}. In Section \ref{generalCategories} we will use the weaker notion of extensions by operations and equations, without new sorts. Now we give toy examples of interpretations.

\begin{example}
The theory of sets with an endofunction is an interpretation of the theory of sets. This just means the extension of:
\begin{eqnarray}
X:\Set
\end{eqnarray}
by:
\begin{eqnarray}
s:X\r X 
\end{eqnarray}
is an interpretation. We do not need any equation defining $s$ because $X$ is empty in the initial model.
\end{example} 

\begin{example}
The theory of groups is an interpretation of the theory of monoids. Indeed it is the extension of the theory of a monoid $(M,e,\cdot)$ where $O$ is:
\begin{eqnarray}
 \_^{-1} &:& M\r M
 \end{eqnarray}
 and $E$ is:
 \begin{eqnarray}
e^{-1} &=& e \label{eqE1}\\ 
(m\cdot n)^{-1} &=& n^{-1}\cdot m^{-1} \label{eqE2} 
\end{eqnarray}
with $E'$ as follows:
\begin{eqnarray}
m\cdot m^{-1} &=& e  \label{eqE'1} \\
m^{-1}\cdot m &=& e \label{eqE'2}
\end{eqnarray}


\end{example}


\begin{example}
The theory of reflexive graphs is an interpretation of the theory of graphs. Indeed this is the extension of the theory of graphs:
\begin{eqnarray}
V &:& \Set\\
E &:& V\r V\r \Set
\end{eqnarray}
by:
\begin{eqnarray}
r &:& (v:V)\r E\ v\ v
\end{eqnarray}
\end{example}

Now we are ready to give our main interpretation.

\begin{theorem}
\label{parametricityAsInterpretation}
Parametricity is an interpretation of type theory.
\end{theorem}
\begin{proof}
We denote by $\TT$ (resp. $\pTT$) the theory of models of type theory (resp. parametric such models). We see that operations $\_^*$ are unary for $\TT$.

In order to prove that $\pTT$ is an interpretation of $\TT$, we need to check that equations (\ref{beginEqParam}) to (\ref{endEqParam}) from Definition \ref{definitionParametricity} indeed define some operations $\_^*$ in $\I_\TT$ the initial model of type theory, as indicated in (\ref{beginOpParam}) to (\ref{endOpParam}). 

To check this we see the initial model of type theory as a QIIT \cite{altenkirch2016type}. 
Then we just need to check that for any equation $s=t$ in $\TT$, the elements $s^*,t^*$ defined recursively by equations (\ref{beginEqParam}) to (\ref{endEqParam}) are equal in $\I_\TT$. This tedious but straightforward task is done in the Appendix.
\end{proof}

\begin{remark}
The reader uninterested in extra generality can replace any interpretation $T'$ of $T$ by the interpretation $\pTT$ of $\TT$ in the rest of this paper. 
\end{remark}

We will show that for any interpretation $T'$ of $T$, the forgetful functor: 
\begin{eqnarray}
U : \Alg_{T'}\r \Alg_T \label{forgetfulFunctor}
\end{eqnarray} 
has a right adjoint. 










\section{Adjoints to forgetful functors}

\label{generalCategories}

In this section we assume an extension of finitary GATs denoted $T\subset T'$ where $T'$ adds operations and equations, but no new sort to $T$. We want to show that the forgetful functor (\ref{forgetfulFunctor}) has a right adjoint if and only if it commutes with finite colimits. We have in mind the special case where $T'$ is an interpretation of $T$.

We will use the fact that finitary GATs correspond to finitary essentially algebraic theories as indicated in Section 6 of \cite{cartmell1986generalised}. This means (among other things) that models for a finitary GAT form an lfp (locally finitely presented) category. The rich theory of such categories is presented in \cite{adamek1994locally}. In this section we will use both GATs and essentially algebraic theories at will. 

All the results presented in this section are well-known to experts in lfp categories. 

\subsection{Existence of a left adjoint}


\begin{lemma}
The forgetful functor:
\begin{eqnarray}
U : \Alg_{T'}\r \Alg_T
\end{eqnarray} 
 has a left adjoint.
\end{lemma}
\begin{proof}
This is mentioned at the end of Section 15 in \cite{cartmell1986generalised}, as a straightforward extension of this fact for algebraic theories \cite{lawvere1963functorial}.
\end{proof}

\begin{example}
\label{degenerateLeftAdjoint}
For parametricity, the left adjoint is freely adding parametricity to a model of type theory. This construction is often degenerate. Indeed freely adding parametricity to a model contradicting parametricity gives an inconsistent model, that is a mathematical world where everything is true. 
\end{example}


\subsection{The forgetful functor is finitary}




Recall that a functor is called finitary if it commutes with filtered colimits, and conservative if it reflects isomorphisms.

\begin{lemma}
\label{finitaryColim}
The forgetful functor:
\begin{eqnarray}
U : \Alg_{T'}\r \Alg_T
\end{eqnarray}
is finitary and conservative.
\end{lemma}
\begin{proof}
Lemma 2.3.6 from \cite{dzierzon2005essentially} states that for a finitary essentially algebraic theory $T$ the forgetful functor (here $S$ is the set of sorts in $T$): 
\begin{eqnarray}
U_T :\mathrm{Alg}_T\r \Set^S
\end{eqnarray} 
is finitary and conservative. Then we have a commuting triangle of functors:
\begin{eqnarray}
U_T U &=& U_{T'}
\end{eqnarray}
and $U$ is finitary and conservative by Lemma \ref{2OutOf2Finitary}.
\end{proof}

\begin{lemma}
\label{2OutOf2Finitary}
Assume two functors $F$ and $G$. If $G$ and $GF$ are finitary and conservative, then so is $F$.
\end{lemma}

\begin{proof}
It is clear that $F$ is conservative, because if $F(g)$ is an isomorphism, so is $GF(g)$, and then so is $g$ because $G F$ is conservative.

For $U : \C\r \D$ a functor and a diagram $i\mapsto c_i$ in $\C$ we denote by $\psi_U$ the canonical map:
\begin{eqnarray}
\psi_U &:& \colim_i\, U(c_i) \r U (\colim_i\, c_i)
\end{eqnarray}
By definition, $U$ commutes with the colimit of $i\mapsto c_i$ if and only if $\psi_U$ is an isomorphism. 

For any diagram $i\mapsto c_i$, we have a commutative triangle: 

\begin{eqnarray}
G(\psi_F)\circ \psi_G = \psi_{GF}
\end{eqnarray}


If the diagram is filtered $\psi_G$ and $\psi_{GF}$ are isomorphisms, therefore so is $G(\psi_F)$. But $G$ is conservative so $\psi_F$ is an isomorphism and $F$ is finitary.
\end{proof}

\subsection{Sufficient condition for a right adjoint}



The next lemma is well-known for functors between lfp categories. 

\begin{lemma}
\label{rightAdjoint}
The forgetful functor:
\begin{eqnarray}
U : \Alg_{T'}\r \Alg_T
\end{eqnarray} 
has a right adjoint if and only if it commutes with small colimits.
\end{lemma}
\begin{proof}
Recall that the category of models for a finitary essentially algebraic theory is an lfp category (by Theorem 3.36 in \cite{adamek1994locally}).
Moreover (by Theorem 1.46 in \cite{adamek1994locally}) any lfp category $\C$ is equivalent to: 
\begin{eqnarray}
\Lex(\C_\fin^{op},\Set)
\end{eqnarray} 
via the Yoneda embedding where: 
\begin{itemize}
\item $\C_\fin$ is the category of finitely presented objects in $\C$.
\item $\Lex(\C,\D)$ is the category of functors from $\C$ to $\D$ commuting with finite limits.
\end{itemize} 

We suppose given a functor between lfp categories:
\begin{eqnarray}
U &:& \C\r\D
\end{eqnarray} 
commuting with small colimits. We define:

\begin{eqnarray}
R &:& \D\r \Lex(\C_\fin^{op},\Set) \\
R(d) &=& \Hom_\D(U\_,d)
\end{eqnarray}

This is well-defined because $U$ commutes with finite colimits. Now recall that in an lfp category any object $c$ is a canonical filtered colimit of finitely presented objects (Proposition 1.22 in \cite{adamek1994locally}), so we have: 
\begin{eqnarray}
c &=& \colim_i\, c_i
\end{eqnarray} 
with $c_i$ finitely presented. But by the definition of $R$ we have: 
\begin{eqnarray}
\Hom_\C(c_i,R(d)) &=& \Hom_\D(U(c_i),d)
\end{eqnarray}
therefore:
\begin{eqnarray}
\Hom_\C(c,R(d)) &=& \Hom_\C(\colim_i\, c_i,R(d)) \\ 
 &=&\lim_i\, \Hom_\C(c_i,R(d))\\
&=& \lim_i\, \Hom_\D(U(c_i),d)\\
&=& \Hom_\D(\colim_i\, U(c_i),d)\\
&=& \Hom_\D(U(c),d)
\end{eqnarray}
where we used the fact that $U$ commutes with filtered colimits. So $R$ is indeed a right adjoint to $U$.
\end{proof}

\begin{theorem}
\label{rightAdjoint2}
The forgetful functor: 
\begin{eqnarray}
U : \Alg_{T'}\r \Alg_T
\end{eqnarray} 
has a right adjoint if and only if it commutes with finite colimits.
\end{theorem}
\begin{proof}
It is well-known that a functor commuting with finite and filtered colimits commutes with small colimits. This is because any small colimit is a coequalizer of coproducts, and a coproduct is a filtered colimit of finite coproducts. 

But $U$ commutes with filtered colimits by Lemma \ref{finitaryColim}, so by Lemma \ref{rightAdjoint} it has a right adjoint if and only if it commutes with finite colimits.
\end{proof}

\section{Constructing right adjoints from interpretations}

\label{constructingRightAdjoint}

In this section we assume given $T'$ an interpretation of $T$. Mimicking parametricity, we denote by $\_^*$ any unary operation added in $T'$. We want to prove that the forgetful functor:
\begin{eqnarray}
U : \Alg_{T'} \r \Alg_T
\end{eqnarray} 
commutes with finite colimits, so that it has a right adjoint. This theorem is proved using the definition of colimits in $\Alg_T$ and $\Alg_{T'}$ as QIITs.

\begin{notation}
In the rest of this section we write $\langle s_j\rangle$ for the sequence $(s_1,\dots,s_n)$, where $n$ can be inferred.
\end{notation}

\begin{notation}
As already indicated in Remark \ref{remarkCStar}, for $c$ a constructor in $T$, and $\_^*$ a unary operation added in $T'$, we denote by: 
\begin{eqnarray}
c\langle x_j\rangle^* &=& c^*\langle x_j,x_j^*\rangle
\end{eqnarray}
with $c^*$ some term in $T$ the equation defining $\_^*$ recursively on $c$. 
\end{notation}


\subsection{Commutation with the initial object} 

The next lemma implies the well-known fact that the initial model of type theory is parametric.

\begin{lemma}
\label{commuteInitial}
The forgetful functor:
\begin{eqnarray}
U : \Alg_{T'}\r \Alg_T
\end{eqnarray} 
commutes with initial objects.
\end{lemma}
\begin{proof}
We consider the initial object $\I_T$ in $\Alg_T$. By definition of an interpretation, we can define operations $\_^*$ on $\I_T$ which obey the unary equations added in $T'$, so that $(\I_T,\_^*) : \Alg_{T'}$. 

Now it is enough to prove that $(\I_T,\_^*)$ is initial in $\Alg_{T'}$, so that: 
\begin{eqnarray}
U(\I_{T'}) = U(\I_T,\_^*) = \I_T
\end{eqnarray}
To this end, we show that the unique morphism: 
\begin{eqnarray}
\psi : \Hom_{\Alg_T}(\I_T,U(\C))
\end{eqnarray} 
for any $\C:\Alg_{T'}$ commutes with $\_^*$. So we prove by induction on $\I_T$ that $\psi(x^*) = \psi(x)^*$ for any $x$.

Indeed we have:
\begin{eqnarray}
\psi(c\langle x_j\rangle^*) &\eqd& \psi(c^*\langle x_j,x_j^*\rangle)\\ 
&\eqm& c^*\langle\psi(x_j),\psi(x_j^*)\rangle \\
&\eqi& c^*\langle\psi(x_j),\psi(x_j)^*\rangle \\
&\eqd& (c\langle\psi(x_j)\rangle)^* \\ 
&\eqm& \psi(c\langle x_j\rangle)^*
\end{eqnarray}
where $\eqi$ indicates induction, $\eqd$ definition of $\_^*$ and $\eqm$ the fact that $\psi$ is a morphism in $\Alg_T$.

From this we can conclude that:
\begin{eqnarray}
\psi &:& \Hom_{\Alg_{T'}}((\I_T,\_^*) , \C)
\end{eqnarray}
and it is the unique such morphism by initiality of $\I_T$.
\end{proof}

\subsection{Commutation with pushouts}

The next lemma uses in a crucial way the hypothesis that operations added in an interpretation are unary.

\begin{lemma}
\label{commutePushout}
The forgetful functor: 
\begin{eqnarray}
U : \Alg_{T'}\r \Alg_T
\end{eqnarray} 
commutes with pushouts.
\end{lemma}
\begin{proof}
Given a span in $\Alg_{T'}$:
\begin{eqnarray}
\C_1\overset{f_1}{\leftarrow} \D \overset{f_2}{\r} \C_2
\end{eqnarray} 
our goal is to define operations $\_^*$ on the pushout:
\begin{eqnarray}
\C = U(\C_1)\coprod_{U(\D)} U(\C_2)
\end{eqnarray}
and prove that $\C$ equipped with these operations is a pushout in $\Alg_{T'}$.

The object $\C$ is generated by:
\begin{itemize} 
\item The constructors in $T$ (as for $\I_T$ in the previous lemma).
\item For $\epsilon=1,2$ morphisms:
\begin{eqnarray}
p_\epsilon &:& \Hom_{\Alg_T}(U(\C_\epsilon),\C)
\end{eqnarray}
meaning that we have constructors $p_\epsilon$ for contexts, types, terms and substitutions, as well as equations for $p_\epsilon$ commuting with any constructor in $T$.
\item For any $x$ in $\D$ we have:
\begin{eqnarray}
p_1(f_1(x)) &=& p_2(f_2(x))
\end{eqnarray}
\end{itemize}

First we define $\_^*$ recursively on $\C$. For the constructors of $T$ we proceed as for $\I_T$, for the new constructors we define:
\begin{eqnarray}
p_\epsilon(x)^* &=& p_\epsilon(x^*)
\end{eqnarray}
This definition makes sense only for unary operations. 

We need to check this preserves the equations. For equations in $T$ this is part of the hypothesis that $\_^*$ is inductively defined, and we see for equations on $p_\epsilon$ that:
\begin{eqnarray}
p_\epsilon(c\langle x_j\rangle)^* &\eqd& p_\epsilon(c\langle x_j\rangle^*) \\
&\eqd& p_\epsilon(c^*\langle x_j,x_j^*\rangle) \\
&\eqm& c^*\langle p_\epsilon(x_j),p_\epsilon(x_j^*)\rangle\\
&\eqd& c^*\langle p_\epsilon(x_j),p_\epsilon(x_j)^*\rangle \\
&\eqd& c\langle p_\epsilon(x_j)\rangle^*
\end{eqnarray}
where $\eqd$ indicates the definition of $\_^*$ and $\eqm$ the fact that $p_\epsilon$ are morphisms in $\Alg_T$. Moreover for $x$ in $\D$:

\begin{eqnarray}
p_1(f_1(x))^* &\eqd& p_1(f_1(x)^*) \\
&\eqf& p_1(f_1(x^*)) \\
&=& p_2(f_2(x^*)) \\
&\eqf& p_2(f_2(x)^*) \\
&\eqd& p_2(f_2(x))^*
\end{eqnarray}
where $\eqd$ indicates the definition of $\_^*$ and $\eqf$ comes from the fact that $f_\epsilon$ is a morphism in $\Alg_{T'}$. So we have defined the operations $\_^*$.

Next we check that any $x$ in $\C$ obeys the unary equations added in $T'$. We proceed inductively on $x$. When $x$ is constructed from $T$ we use the hypothesis that equations are inductively proven in the initial model. When $x$ is of the form $p_\epsilon(x')$ we use the fact that equations added in $T'$ are true in $\C_\epsilon$ so they are true for $x'$, together with the fact that $p_\epsilon$ is a morphism so it preserves equations.

Now we have $(\C,\_^*):\Alg_{T'}$, we want to check that it is a pushout. To do this it is enough to check that any morphism:
\begin{eqnarray}
\psi &:& \Hom_{\Alg_T}(U(\C_1)\coprod_{U(\D)}U(\C_2) , U(\E))
\end{eqnarray}
defined from a commutative square in $\Alg_{T'}$:
\[\centerline{\xymatrix{
\D \ar[d]_{f_2}\ar[r]^{f_1}& \C_1 \ar[d]^{g_1}\\
\C_2 \ar[r]_{g_2} & \E
}} \]
does commute with the operations $\_^*$ previously defined. By definition, $\psi$ is such that: 
\begin{eqnarray}
\psi(p_\epsilon(x)) = g_\epsilon(x)
\end{eqnarray} 

We proceed inductively, using computations from the previous lemma, together with the following new case:
\begin{eqnarray}
\psi(p_\epsilon(x)^*) &=& \psi(p_\epsilon(x^*))\\
 &=& g_\epsilon(x^*) \\
 &=& g_\epsilon(x)^* \\
 &=& \psi(p_\epsilon(x))^*
\end{eqnarray}

This concludes the proof.
\end{proof}

\subsection{Main theorem and applications}

We are ready to give the most important result in this paper. Recall that we have assumed a forgetful functor: 
\begin{eqnarray}
U:\Alg_{T'}\r\Alg_T
\end{eqnarray}
for $T'$ an interpretation of $T$.

\begin{theorem}
\label{mainTheorem}
The forgetful functor: 
\begin{eqnarray}
U : \Alg_{T'}\r \Alg_T
\end{eqnarray} 
has a right adjoint.
\end{theorem}
\begin{proof}
By Theorem \ref{rightAdjoint2} it is enough to show that $U$ commutes with finite colimits. This is precisely the content of Lemmas \ref{commuteInitial} and \ref{commutePushout}.
\end{proof}

\begin{example}
We consider groups interpreting monoids. In this case the right adjoint is:
\begin{eqnarray}
C &:& \mathrm{Mon} \r \mathrm{Grp} \\
C(M) &=& M^\times
\end{eqnarray}
where $M^\times$ is the group of invertible elements in $M$. Indeed $\mathbb{Z}$ is the free group generated by $\{1\}$, so the underlying set of $C(M)$ is:
\begin{eqnarray} 
C(M) &=& \Hom_\Set(\{1\},C(M)) \\
&=& \Hom_\mathrm{Grp}(\mathbb{Z},C(M)) \\
&=& \Hom_\mathrm{Mon}(\mathbb{Z},M)\\
&=& M^\times \label{equationFreeGroup}
\end{eqnarray} 
The group structure is computed in the same way.
\end{example}

\begin{example}
We consider the interpretation of graphs by reflexive graphs. Recall that reflexive graphs are given by the theory:
\begin{eqnarray}
V &:& \Set \\ 
E &:& V\r V\r\Set \\
 r &:& (v:V)\r E\ v\ v
\end{eqnarray}
The right adjoint is:
\begin{eqnarray}
C &:& \mathrm{Gph}\r \mathrm{rGph} \\
C(V,E) &=& (v:V)\times E\ v\ v , \nonumber \\ &&  (v,e)(v',e') \mapsto E\ v\ v' , \nonumber\\ && (v,e)\mapsto e
\end{eqnarray}
To see this we consider $\I_c$ the free reflexive graph generated by $\{\mathrm{c}\}$ a vertex. Then the set of vertices of $C(V,E)$ is:
\begin{eqnarray}
C(V,E) &=& \Hom_\Set(\{\mathrm{c}\},C(V,E)) \\
&=& \Hom_{\mathrm{rGph}}(\I_c,C(V,E)) \\
&=& \Hom_\mathrm{Gph}(U(\I_c),(V,E)) \\
&=& (v:V)\times E\ v\ v
\end{eqnarray}
The rest of the structure is computed in the same way.
\end{example}

\begin{example}
We consider the interpretation:
\begin{eqnarray}
X &:& \Set \\ 
s &:& X\r X
\end{eqnarray}
of the theory with $X:\Set$ alone. We denote its category of models by $\Set_s$.
\begin{eqnarray} 
\Set_s &=&\{X:\Set\ |\ s:X\r X\}
\end{eqnarray}
Then the right adjoint is:
\begin{eqnarray}
C &:& \Set \r \Set_s \\
C(X) &=& \mathbb{N} \r X, \nonumber\\ && f\mapsto (n\mapsto f(n+1))
\end{eqnarray}
To see this, note that $\mathbb{N}$ with the successor function is the free object in $\Set_s$ generated by $\{0\}$. Then the underlying set of $C(X)$ is:
\begin{eqnarray}
C(X) &=& \Hom_\Set(\{0\},C(X)) \\
&=& \Hom_{\Set_s}(\mathbb{N},C(X)) \\
&=& \Hom_\Set(\mathbb{N},X)
\end{eqnarray}
The function $f \mapsto ( n \mapsto f(n+1))$ is computed in the same way.
\end{example}

These three examples should be contrasted with each other:
\begin{itemize}
\item In the first example, being invertible is a property of an element in a monoid (because there is at most one inverse). Then the right adjoint just needs to send a monoid to its group of invertible elements. This can be generalized to any unary property inductively provable, with the right adjoint sending an object to its subobject of elements obeying this property.
\item In the second example, having an edge from $v$ to $v$ is really a structure on a vertex $v$ (because there can be many such edges). So in this case we need to consider vertices together with a chosen edge in order to construct the right adjoint.
\item The third example is the most interesting. Here having an image by $s$ is clearly a structure, but to build the right adjoint it is not enough to require that any element comes with its image by $s$. Indeed this image should itself have an image, and so on. An iteration is taking place. This can be generalized to the interpretation of any theory $T$ by the theory of $T$-algebras with an endomorphism.
\end{itemize}

Now we study our main example.

\begin{example}
\label{remarkUnary}
Consider the forgetful functor:
\begin{eqnarray}
U : \Alg_{\pTT}\r \Alg_\TT
\end{eqnarray}
from parametric models to models of type theory. By Theorems \ref{parametricityAsInterpretation} and \ref{mainTheorem} it has a right adjoint $C$. Now we study $C(\D)$ for $\D$ a model of type theory. We denote by $\I_X$ the free parametric model generated by an element $X:\Ctx$. Then:
\begin{eqnarray}
\Ctx_{C(\D)} &=& \Hom_{\Set}(\{X\},\Ctx_{C(\D)})\\
&=& \Hom_{\Alg_{\pTT}}(\I_X,C(\D)) \\
&=&\Hom_{\Alg_\TT}(U(\I_X),\D)
\end{eqnarray}
but reasoning as in Lemmas \ref{commuteInitial} and \ref{commutePushout}, we can prove that $U(\I_X)$ is isomorphic to the free (non-parametric) model of type theory with:
\begin{eqnarray}
X &:& \Ctx\\
X^* &:& \Ty\ X\\
X^{**} &:& \Ty\ (x:X,X^*(x),X^*(x))\\
&\vdots& \nonumber 
\end{eqnarray}
with the usual notation for contexts in a CwF. So giving a context in $C(\D)$
is equivalent to giving:
\begin{eqnarray}
\Gamma &:& \Ctx_\D\\
\Gamma^* &:& \Ty_\D\ \Gamma\\
\Gamma^{**} &:& \Ty_\D\ (x:\Gamma,\Gamma^*(x),\Gamma^*(x))\\
&\vdots& \nonumber
\end{eqnarray}
 We can get similar formulas for types, terms, and so on.
\end{example}

\begin{remark}
The right adjoint $C$ does not suffer from the same defect as the left adjoint in Example \ref{degenerateLeftAdjoint}. We assume an empty type $\bot$ and we say that a model is inconsistent if its $\bot$ is inhabited. If $C(\D)$ is inconsistent then so is $U(C(\D))$, and the counit: 
\begin{eqnarray}
\epsilon &:& \Hom_{\Alg_\TT}(U(C(\D)),\D)
\end{eqnarray} 
implies that $\D$ is inconsistent.
\end{remark}

\begin{remark}
\label{remarkBinary}
Binary parametricity can be treated by our method. By following the last example, we see that a context in $C(\D)$ would consist of:
\begin{eqnarray}
\Gamma &:& \Ctx_\D\\
\Gamma^* &:& \Ty_\D\ (\Gamma,\Gamma)\\
\Gamma^{**} &:& \Ty_\D\ (x_{00},x_{01}:\Gamma,\Gamma^*(x_{00},x_{01}), \nonumber\\
 & & \hspace{27pt} x_{10},x_{11}:\Gamma,\Gamma^*(x_{10},x_{11}), \nonumber\\
  & & \hspace{27pt} \Gamma^*(x_{00},x_{10}),\Gamma^*(x_{01},x_{11})) \\
  &\vdots& \nonumber
\end{eqnarray}
So we are constructing the semi-cubical model in $\D$.
\end{remark}

\begin{remark}
It should be noted that our method does not immediately give an explicit definition for the three dots in Remarks \ref{remarkUnary} and \ref{remarkBinary}. This lack of concreteness is compensated by some extra generality. 
\end{remark}

\section{Conclusion}

In this paper we defined a procedure from interpretations of type theory to structures on types, outputting semi-cubical structure when given external parametricity. In order to do this we defined interpretations of any theory, and built a right adjoint from any such interpretation.

Our next work will be to apply this method to other interpretations of type theory. We would like to study forcing interpretations giving us some variant of presheaf models, and also the \emph{univalent parametricity} from \cite{tabareau2018equivalences}, hopefully allowing us to build definitionally univalent models (meaning models where $A=_\U B$ is definitionally equal to $A\simeq B$) from univalent models.

We would also like to make sense of the table already given in introduction:
\begin{center}
\begin{tabular}{|c|c|c|}
\hline
Interpretation & Structure \\
\hline
External parametricity & Semi-cubical types\\
Internal parametricity & Cubical types\\
External univalence & Kan semi-cubical types\\
Internal univalence & Kan cubical types\\
\hline
\end{tabular}
\end{center}
This means that we need to find suitable interpretations. From a practical point of view this would give an efficient way to build univalent models of type theory, and might help to design variants of cubical type theory. From a conceptual point of view this could explain how the notion of Kan cubical structure can be deduced from the notion of equivalence.

We also believe this work shed some light on forgetful functors having both a left and a right adjoint. There is a large literature on forgetful functors having a right adjoint and coalgebras (see for example \cite{adamek2003varieties}), but we do not know any reference on such functors having a left adjoint as well. We believe these could be called unary functors by analogy with finitary functors. More precisely, we guess there is some kind of converse to Theorem \ref{mainTheorem}, stating that the forgetful functor of an extension of theories has a right adjoint if and only if the extension obeys a condition weaker than being an interpretation, and stronger than having unary operations.

\section*{Acknowledgements}

We would like to thank the anonymous reviewers for their many interesting remarks and suggestions. We are also grateful to Ambrus Kaposi for suggesting Remark \ref{definitionInterpretation}. Finally we are thankful to Hugo Herbelin for countless mathematical discussions on parametricity, cubes and related subjects.

\bibliographystyle{plain}
\bibliography{Semi_cubical_arxiv}

\break

\section{Appendix: Checking parametricity is well-defined}



Our goal here is to check that given any equation $s=t$ in $\TT$, the inductive definitions of $s^*$ and $t^*$ are equal in $\TT$. We use the symbol $\_\equiv\_$ for equality by definition of $\_^*$, and $\_=\_$ for equality in $\TT$.

First we check equations for the calculus of substitutions.

\begin{eqnarray}
((\sigma\circ \nu)\circ \delta)^* 
\equiv \sigma^*[\nu\circ\wk,\nu^*][\delta\circ\wk,\delta^*] &=& \sigma^*[\nu\circ\delta\circ\wk,\nu^*[\delta\circ\wk,\delta^*]] 
\equiv (\sigma\circ(\nu\circ\delta))^*  \nonumber \\
(\id\circ\sigma)^* 
\equiv \var[\wk,\sigma^*] &=& \sigma^* \nonumber \\
(\sigma\circ\id)^* 
\equiv \sigma^*[\wk,\var] &=& \sigma^* \nonumber \\
\epsilon^* \equiv \t &=& \sigma^* \ \mathrm{for}\ \sigma : \Sub\ \Gamma\ \emptyctx \nonumber\\
(\pi_1\ (\sigma,t))^* 
\equiv (\sigma^*,t^*).1 &=& \sigma^* \nonumber\\
(\pi_2\ (\sigma,t))^* 
\equiv (\sigma^*,t^*).2 &=& t^* \nonumber\\
(\pi_1\ \sigma , \pi_2\ \sigma)^* 
\equiv (\sigma^*.1,\sigma^*.2) &=& \sigma^* \nonumber\\
((\sigma,t)\circ \nu)^* 
\equiv (\sigma^*,t^*)[\nu\circ\wk,\nu^*] &=& (\sigma^*[\nu\circ\wk,\nu^*],t^*[\nu\circ\wk,\nu^*]) 
\equiv (\sigma\circ\nu,t[\nu])^* \nonumber
\end{eqnarray}

Next we check equations for the unit.

\begin{eqnarray}
\t^*\equiv \t &=& x^*\ \mathrm{for}\ x:\Tm\ \Gamma\ \top \nonumber\\
(\top[\sigma])^* 
\equiv \top[\sigma\circ\wk^2,\sigma^*[\wk],\var] &=& \top \equiv \top^*\nonumber \\
(\t[\sigma])^* 
\equiv \t[\sigma\circ\wk,\sigma^*] &=& \t \equiv \t^*\nonumber
\end{eqnarray}

Now we check equations for products. The constructor $\Sigma$ is our first cumbersome case:

\begin{eqnarray}
((\Sigma\ A\ B)[\sigma])^* 
&\equiv& (\Sigma\ (A^*[\wk,\var.1])\ (B^*[\wk^3,\var.1[\wk],(\var[\wk^2],\var),\var.2[\wk]]))[\sigma\circ\wk^2,\sigma^*[\wk],\var] \nonumber\\
&=& \Sigma\ (A^*[\wk,\var.1][\sigma\circ\wk^2,\sigma^*[\wk],\var])\ (B^* [\wk^3,\var.1[\wk],(\var[\wk^2],\var),\var.2[\wk]][\sigma\circ\wk^3,\sigma^*[\wk^2],\var[\wk],\var]) \nonumber\\
&=& \Sigma\ (A^*[\sigma\circ\wk^2,\sigma^*[\wk],\var.1])\ (B^*[\sigma\circ\wk^3,\var.1[\wk],(\sigma^*[\wk^2],\var),\var.2[\wk]]) \nonumber\\
&=& \Sigma\ (A^*[\sigma\circ\wk^2,\sigma^*[\wk],\var][\wk,\var.1])\ \nonumber\\
&& \ \ \ (B^*[\sigma\circ\wk^3,\var[\wk^2],(\sigma^*[\wk^2,\var.1],\var.2)[\wk],\var][\wk^3,\var.1[\wk],(\var[\wk^2],\var),\var.2[\wk]]) \nonumber\\
&\equiv& 
(\Sigma\ A[\sigma]\ B[\sigma\circ\wk,\var])^* \nonumber
\end{eqnarray}


Now for terms we have:

\begin{eqnarray}
((s,t).1)^* 
\equiv (s^*,t^*).1 &=& s^* \nonumber\\
((s,t).2)^* 
\equiv (s^*,t^*).2 &=& t^* \nonumber\\
(t.1,t.2)^* 
\equiv (t^*.1,t^*.2) &=& t^* \nonumber\\
((s,t)[\sigma])^* 
\equiv (s^*,t^*)[\sigma\circ\wk,\sigma^*] &=& (s^*[\sigma\circ\wk,\sigma^*],t^*[\sigma\circ\wk,\sigma^*]) 
\equiv (s[\sigma],t[\sigma])^* \nonumber
\end{eqnarray}


Now we proceed with equations for functions. The case of $\Pi$ is complicated to write down.

\begin{eqnarray}
((\Pi\ A\ B)[\sigma])^* 
 &\equiv& (\Pi\ A[\wk^2]\ (\Pi\ A^*[\wk^2,\var]\ B^*[\wk^4,\var[\wk],(\var[\wk^3],\var),(\app\ \var)[\wk]]))[\sigma\circ\wk^2,\sigma^*[\wk],\var] \nonumber\\
&=& \Pi\ A[\wk^2][\sigma\circ\wk^2,\sigma^*[\wk],\var]\ (\Pi\ A^*[\wk^2,\var][\sigma\circ\wk^3,\sigma^*[\wk^2],\var[\wk],\var]\ \nonumber\\
& & \ \ \  (B^*[\wk^4,\var[\wk],(\var[\wk^3],\var),(\app\ \var)[\wk]][\sigma\circ\wk^4,\sigma^*[\wk^3],\var[\wk^2],\var[\wk],\var])) \nonumber\\
&=& \Pi\ A[\sigma\circ\wk^2]\ (\Pi\ A^*[\sigma\circ\wk^3,\sigma^*[\wk^2],\var]\ (B^*[\sigma\circ\wk^4,\var[\wk],(\sigma^*[\wk^3],\var),(\app\ \var)[\wk]]))\nonumber\\
&=& \Pi\ A[\sigma\circ\wk^2]\ (\Pi\ A^*[\sigma\circ\wk^2,\sigma^*[\wk],\var][\wk^2,\var]\ \nonumber\\
& & \ \ \  (B^*[\sigma\circ\wk^3,\var[\wk^2],(\sigma^*[\wk^2,\var.1],\var.2)[\wk],\var][\wk^4,\var[\wk],(\var[\wk^3],\var),(\app\ \var)[\wk]])) \nonumber\\
 &\equiv& (\Pi\ A[\sigma]\ B[\sigma\circ\wk,\var])^*\nonumber
\end{eqnarray}


Then for substitution in $\lambda$ we have:
\begin{eqnarray}
((\lambda\ t)[\sigma])^* 
&\equiv& \lambda\ (\lambda\ (t^*[\wk^3,\var[\wk],(\var[\wk^2],\var)]))[\sigma\circ\wk,\sigma^*] \nonumber\\
&=& \lambda\ (\lambda\ (t^*[\wk^3,\var[\wk],(\var[\wk^2],\var)][\sigma\circ\wk^3,\sigma^*[\wk^2],\var[\wk],\var])) \nonumber\\
&=& \lambda\ (\lambda\ (t^*[\sigma\circ\wk^3,\var[\wk],(\sigma^*[\wk^2],\var)]))\nonumber\\
&=& \lambda\ (\lambda\ (t^*[\sigma\circ\wk^3,\var[\wk],\sigma^*[\wk^3,\var[\wk^2],\var])\nonumber\\
&=& \lambda\ (\lambda\ (t^*[\sigma\circ\wk^2,\var[\wk],\sigma^*[\wk^2,\var.1,\var.2]][\wk^3,\var[\wk],(\var[\wk^2],\var)]))\nonumber\\
&\equiv& (\lambda\ (t[\sigma\circ\wk,\var]))^*\nonumber
\end{eqnarray}

And for $\app$ and $\lambda$ composed one way:
\begin{eqnarray}
(\app\ (\lambda\ t))^* 
&\equiv& \app\ (\app\ (\lambda\ (\lambda\ t^*[\wk^3,\var[\wk],(\var[\wk^2],\var)])))[\wk^2,\var.1,\var[\wk],\var.2] \nonumber\\
&=& t^*[\wk^3,\var[\wk],(\var[\wk^2],\var)][\wk^2,\var.1,\var[\wk],\var.2] \nonumber\\
&=& t^*[\wk^2,\var[\wk],(\var.1,\var.2)] \nonumber\\
&=& t^*[\wk^2,\var[\wk],\var] \nonumber\\
&=& t^* \nonumber
\end{eqnarray}

And the other:

\begin{eqnarray}
(\lambda\ (\app\ t))^* 
 &\equiv& \lambda\ (\lambda\ (\app\ (\app\ t^*)[\wk^2,\var.1,\var[\wk],\var.2][\wk^3,\var[\wk],(\var[\wk^2],\var)])) \nonumber\\
 &=& \lambda\ (\lambda\ (\app\ (\app\ t^*)[\wk^3,\var[\wk^2],\var[\wk],\var]))\nonumber\\ 
 &=& \lambda\ (\lambda\ (\app\ (\app\ t^*))) \nonumber\\
 &=& t^*\nonumber
\end{eqnarray}


And last we check equations for the universe. First for substitutions in $\U$, $\El$ and $\top_\U$ we have:

\begin{eqnarray}
(\U[\sigma])^* 
\equiv (\Pi\ (\El\ \var)\ \U)[\sigma\circ\wk^2,\sigma^*[\wk],\var] &=& \Pi\ (\El\ \var)\ \U \equiv \U^*  \nonumber\\
((\El\ t)[\sigma])^* 
\equiv (\El\ (\app\ t^*))[\sigma\circ\wk^2,\sigma^*[\wk],\var] &=& \El\ (\app\ (t^*[\sigma\circ\wk,\sigma^*])) 
\equiv (\El\ (t[\sigma]))^*  \nonumber\\
(\top_\U[\sigma])^* 
\equiv (\lambda\ \top_\U)[\sigma\circ\wk,\sigma^*] &=& \lambda\ \top_\U \equiv \top_\U^*  \nonumber
\end{eqnarray}

Now for substitution in $\Sigma_\U$ we have:

\begin{eqnarray}
((\Sigma_\U\ s\ t)[\sigma])^* 
&\equiv& (\lambda\ (\Sigma_\U\ (\app\ s^*)[\wk,\var.1]\ (\app\ t^*)[\wk^3,\var.1[\wk],(\var[\wk^2],\var),\var.2[\wk]]))[\sigma\circ\wk,\sigma^*]  \nonumber\\
&=&\lambda\ (\Sigma_\U\ (\app\ s^*)[\wk,\var.1][\sigma\circ\wk^2,\sigma^*[\wk],\var]\ \nonumber\\
&& \ \ \ (\app\ t^*)[\wk^3,\var.1[\wk],(\var[\wk^2],\var),\var.2[\wk]][\sigma\circ\wk^3,\sigma^*[\wk^2],\var[\wk],\var]) \nonumber\\
&=& \lambda\ (\Sigma_\U\ (\app\ s^*)[\sigma\circ\wk^2,\sigma^*[\wk],\var.1]\ (\app\ t^*)[\sigma\circ\wk^3,\var.1[\wk],(\sigma^*[\wk^2],\var),\var.2[\wk])  \nonumber\\
 &=& \lambda\ (\Sigma_\U\ (\app\ s^*)[\sigma\circ\wk^2,\sigma^*[\wk],\var][\wk,\var.1]\ \nonumber\\
 & & \ \ \ (\app\ t^*)[\sigma\circ\wk^3,\var[\wk^2],(\sigma^*[\wk^2,\var.1],\var.2)[\wk],\var][\wk^3,\var.1[\wk],(\var[\wk^2],\var),\var.2[\wk]])  \nonumber\\
&\equiv& (\Sigma_\U\ s[\sigma]\ t[\sigma\circ\wk,\var])^*  \nonumber
\end{eqnarray}



And for substitution in $\Pi_\U$ we have:

\begin{eqnarray}
 ((\Pi_\U\ s\ t)[\sigma])^* 
&\equiv& \lambda\ (\Pi_\U\ s[\wk^2]\ (\Pi_\U\ (\app\ s^*)[\wk^2,\var]\ (\app\ t^*)[\wk^4,\var[\wk],(\var[\wk^3],\var),(\app\ \var)[\wk]]))[\sigma\circ\wk,\sigma^*] \nonumber\\
&=& \lambda\ (\Pi_\U\ s[\wk^2][\sigma\circ\wk^2,\sigma^*[\wk],\var]\ (\Pi_\U\ (\app\ s^*)[\wk^2,\var][\sigma\circ\wk^3,\sigma^*[\wk^2],\var[\wk],\var]\  \nonumber\\
& & \ \ \ (\app\ t^*)[\wk^4,\var[\wk],(\var[\wk^3],\var),(\app\ \var)[\wk]][\sigma\circ\wk^4,\sigma^*[\wk^3],\var[\wk^2],\var[\wk],\var])) \nonumber\\
&=& \lambda\ (\Pi_\U\ s[\sigma\circ\wk^2]\ (\Pi_\U\ (\app\ s^*)[\sigma\circ\wk^3,\sigma^*[\wk^2],\var]\ \nonumber\\
&& \ \ \ (\app\ t^*)[\sigma\circ\wk^4,\var[\wk],(\sigma^*[\wk^3],\var),(\app\ \var)[\wk]])) \nonumber\\
&=& \lambda\ (\Pi_\U\ s[\sigma\circ\wk^2]\ (\Pi_\U\ (\app\ s^*)[\sigma\circ\wk^2,\sigma^*[\wk],\var][\wk^2,\var]\  \nonumber\\
& & \ \ \ (\app\ t^*)[\sigma\circ\wk^3,\var[\wk^2],(\sigma^*[\wk^2,\var.1],\var.2)[\wk],\var][\wk^4,\var[\wk],(\var[\wk^3],\var),(\app\ \var)[\wk]])) \nonumber\\
 &\equiv& (\Pi_\U\ s[\sigma]\ t[\sigma\circ\wk,\var])^* \nonumber
\end{eqnarray}



And finally we can check the equation for $\top_\U$ as follows:

\begin{eqnarray}
(\El\ \top_\U)^* 
\equiv \El\ (\app\ (\lambda\ \top_\U)) &=& 
\top \equiv \top^*  \nonumber
\end{eqnarray}

And then for $\Sigma_\U$ and $\Pi_\U$ we have:

\begin{eqnarray}
(\El\ (\Sigma_\U\ s\ t))^* 
&\equiv& \El\ (\app\ (\lambda\ (\Sigma_\U\ (\app\ s^*)[\eta_1]\ (\app\ t^*)[\eta_2])))  \nonumber\\
&=& \Sigma\ (\El\ (\app\ s^*))[\eta_1]\ (\El\ (\app\ t^*))[\eta_2] \nonumber\\
&\equiv& (\Sigma\ (\El\ s)\ (\El\ t))^*  \nonumber\\
(\El\ (\Pi_\U\ s\ t))^* 
 &\equiv& \El\ (\app\ (\lambda\ (\Pi_\U\ s[\sigma_1]\ (\Pi_\U\ (\app\ s^*)[\sigma_2]\ (\app\ t^*)[\sigma_3])))  \nonumber\\
&=& \Pi\ (\El\ s)[\sigma_1]\ (\Pi\ (\El\ (\app\ s^*))[\sigma_2]\ (\El\ (\app\ t^*))[\sigma_3])  \nonumber\\
&\equiv& (\Pi\ (\El\ s) (\El\ t))^*
 \nonumber
\end{eqnarray}

Where we have: 

\begin{eqnarray}
\eta_1 &=& (\wk,\var.1) \nonumber\\
\eta_2 &=& (\wk^3,\var.1[\wk],(\var[\wk^2],\var),\var.2[\wk]) \nonumber\\
\sigma_1 &=& \wk^2 \nonumber\\
\sigma_2 &=& (\wk^2,\var) \nonumber\\
\sigma_3 &=& (\wk^4,\var[\wk],(\var[\wk^3],\var),(\app\ \var)[\wk]) \nonumber
\end{eqnarray}

\end{document}